\documentclass[12pt]{amsart}
\usepackage{mathptmx}
\usepackage{helvet}
\usepackage{courier}
\usepackage[shortlabels]{enumitem}
\usepackage{type1cm}         

\usepackage{graphicx}        
\usepackage{multicol}        
\usepackage[bottom]{footmisc}
\usepackage{glossaries}
\usepackage{imakeidx}

\usepackage{amsmath,amssymb,amsfonts}%
\usepackage{bm}				
\usepackage{placeins} 		
\usepackage{tikz} 
\usepackage{hyperref}
\usepackage{wasysym}
\usetikzlibrary{decorations.pathmorphing}

\DeclareMathOperator{\dom}{dom}

\newcounter{my_enumerate_counter}
\newcommand{\pushcounter}{\setcounter{my_enumerate_counter}{\value{enumi}}}
\newcommand{\popcounter}{\setcounter{enumi}{\value{my_enumerate_counter}}}

\DeclareMathOperator{\Th}{Th}

\newcommand{\cP}{\mathcal P}

\newcommand{\bbC}{{\mathbb C}}

\newcommand{\bbN}{{\mathbb N}}

\newcommand{\calL}{{\mathcal L}}

\newcommand{\cU}{{\mathcal U}}

\newcommand{\tr}{{\rm tr}}

\newcommand{\cstar}{$\mathrm{C}^*$}

\DeclareMathOperator{\eq}{eq}

\newcommand{\bbM}{{\mathbb M}}

\newcommand{\bbF}{{\mathbb F}}

\newcommand{\bbQ}{\mathbb Q}

\newcommand{\cI}{{\mathcal I}}
\newcommand{\calR}{{\mathcal R}}

\newcommand{\cX}{{\mathcal X}}

\newcommand{\cY}{{\mathcal Y}}

\newcommand{\cC}{\mathcal C}

\newcommand{\cB}{\mathcal B}

\newcommand{\cM}{\mathcal M}

\newcommand{\dminus}{\dot -}

\newtheorem{thm}{Theorem}
\newtheorem{theorem}{Theorem}[section] 
\newtheorem{corollary}[theorem]{Corollary}

\newtheorem{conjecture}[theorem]{Conjecture}

\newtheorem{claim}[theorem]{Claim}

\newtheorem{lemma}[theorem]{Lemma}
\newtheorem{proposition}[theorem]{Proposition}

\theoremstyle{definition}

\newtheorem{definition}[theorem]{Definition}

\DeclareMathOperator{\trm}{tr_m}

\DeclareMathOperator{\LMBA}{\calL_{\text{MBA}}}
\newcommand{\botimes}{\bar\otimes} 
\DeclareMathOperator{\frFx}{\mathfrak F^{\bar x}}

\newcommand{\bbTD}{\mathbb T_D}

\newcommand{\LLONE}{\mathcal L_{\|\cdot\|_1}}
\newcommand{\LLTWO}{\mathcal L_{\|\cdot\|_2}}
\numberwithin{equation}{section}

\begin{document}


\title[Preservation of elementarity\dots]{Preservation of elementarity by tensor products of tracial von Neumann algebras}
\date{\today}

\author{Ilijas Farah}

\address{Department of Mathematics and Statistics\\
	York University\\
	4700 Keele Street\\
	North York, Ontario\\ Canada, M3J 1P3\\
	and 
	Matemati\v cki Institut SANU\\
	Kneza Mihaila 36\\
	11\,000 Beograd, p.p. 367\\
	Serbia}
\email{ifarah@yorku.ca}
\urladdr{https://ifarah.mathstats.yorku.ca}

\author{Saeed Ghasemi}

\address{Department of Mathematics and Statistics\\
	York University\\
	4700 Keele Street\\
	North York, Ontario\\ Canada, M3J 1P3\\
	and 
	Lakehead University\\
	955 Oliver Road
	Thunder Bay, Ontario\\
	Canada P7B 5E1 (current address)}
\urladdr{https://www.lakeheadu.ca/users/G/sghasem2/node/201808}
\email{sghasem2@lakeheadu.ca}
\maketitle

\begin{abstract}
	Tensoring with type I algebras preserves elementary equivalence in the category of tracial von Neumann algebras. The proof involves a novel and general Feferman--Vaught-type theorem for direct integrals of metric structures. 
\end{abstract}
In \cite{feferman1959first} it was proven that reduced products and their generalizations preserve elementary equivalence, in the sense that the first-order theory of the product can be computed from the theories of the factors and information about the ideal (if any) used to form the reduced product. The question of preservation of elementarity by tensor products and free products is a bit subtler. Somewhat surprisingly, free products preserve elementary equivalence both in the case of groupoids (\cite{olin1974free}) and, by a deep result of Sela, in the case of groups (\cite{sela2010diophantine}). It is not known whether this is the case with tracial von Neumann algebras (\cite[Question~5.3]{jekel2022free}). 

On the other hand, tensor products in general do not preserve elementary equivalence in the category of modules (\cite{Olin:Direct}) and in the category of \cstar-algebras (\cite{Muenster}, \cite{farah2022between}).  
David Jekel asked whether tensor products of tracial von Neumann algebras preserve elementary equivalence (\cite[\S 5.1]{jekel2022free}). We give  partial positive answers to this question (see \S\ref{S.elementarity} for the notation and terminology). 

\begin{thm} \label{T.tensor} If $M$ and $N$ are tracial von Neumann algebras at least one of which is type I, then the theory of their tensor product depends only on theories of $M$ and $N$. 
	
	In other words, if $M_1\equiv M$ and $N_1\equiv N$ then $M_1\bar\otimes N_1\equiv M\otimes N$. 
	
	More precisely, if $M_1\preceq M$ and $N_1\preceq N$ then (with the natural identification) $M_1\bar\otimes N_1\preceq M\bar \otimes N$. 
\end{thm}

In the course of proving Theorem~\ref{T.tensor} we prove a Feferman--Vaught type theorem for direct integrals of metric structures (Theorem~\ref{T.FV-integral}). This proof roughly follows the lines of the proof of the Feferman--Vaught theorem for metric reduced products given in \cite{ghasemi2016reduced} (see also \cite[\S 16]{Fa:STCstar}).
Also, standard results imply that among McDuff factors, tensoring with the hyperfinite II$_1$ factor preserves elementarity (see \S\ref{S.concluding}). 
Unlike \cstar-algebras, among tracial von Neumann algebras there is no known example of a failure of preservation of elementarity by tensor products. 

\subsection*{Acknowledgements} We would like to thank participants of the  reading seminar on free probability and model theory held at York University in the fall semester of 2022 and to Tom\'as Ibarlucia for a very useful remark. We are also indebted to Ita\"  i Ben Ya'acov,  David Gao, David Jekel, and the anonymous referees for remarks on an earlier draft on this paper. 

\section{Preliminaries}

Good  general references are, for operator algebras~\cite{Black:Operator}, for type II$_1$ factors~\cite{anantharaman2017introduction},  
		  and for continuous model theory \cite{BYBHU}, \cite{hart2023introduction}, \cite{goldbring2022survey}.

\subsection{Direct integrals}
Our original definition of measurable fields of metric structures and direct integrals of metric structures was analogous to measurable fields and  direct integrals of Hilbert spaces and of von Neumann algebras (\cite[Definitions~IV.8.9, IV.8.15, and IV.8.17]{Tak:TheoryI}). Here we include the more polished definition from~\cite[\S 8]{yaacov2024extremal}.\footnote{A visible, but ultimately insubstantial, technical difference between this definition and the standard definition of a direct integral of tracial von Neumann algebras will be discussed in \S\ref{S.vNA}.}  Although this definition can be extended to the case of nonseparable metric spaces, we will consider only the separable case. For simplicity, we will restrict our attention to the case when $\calL$ is a single-sorted language; the definition is generalized to multi-sorted languages by making obvious modifications. 

\begin{definition}[Measurable fields and direct integrals of metric structures] \label{Def.DirectIntegral}
	Suppose that $\calL$ is a continuous language, $(\Omega,\cB,\mu)$ is a separable probability measure space, and $(M_\omega, d_\omega)$, for $\omega\in \Omega$, are separable $\calL$-structures. Assume that $e_n$, for $n\in \bbN$, is a sequence in $\prod_{\omega\in \Omega} M_\omega$ such that the following two conditions hold. 
	\begin{enumerate}
		\item for every $\omega$ the set $\{e_n(\omega)\mid n\in \bbN\}$ is dense in $M_\omega$. 
		\item\label{2.Def.DirectIntegral} For every predicate $R(\bar x)$ in $\calL$ and every tuple $e_{\bar n}=\langle e_{n(0)}, \dots, e_{n(j-1)}\rangle$ of the appropriate sort, the function $\omega\mapsto R^{M_\omega}(e_{\bar n}(\omega))$ is measurable. 
	\end{enumerate}
	The structures $M_\omega$, together with the functions $e_n$, form a \emph{measurable field} of $\calL$-structures. 
	
	The \emph{direct integral of $M_\omega$, for $\omega\in \Omega$}, is the structure denoted 
	\[
	M=\int^\oplus_\Omega M_\omega\, d\mu(\omega)
	\] 
	and defined as follows. 
		Consider the set $\cM$ of all $a\in \prod_{\omega\in \Omega} M_\omega$ such that the functions
	\begin{align*}
		\omega&\mapsto d_\omega(a(\omega),e_n(\omega))
	\end{align*} 
	are measurable for all $n$. On $\cM$ consider the pseudo-metric 
	\[
	d^M(a,b)=\int_\Omega d_\omega(a(\omega),b(\omega))\, d\mu(\omega).
	\] 
	Then the domain of $M$ is defined to be the set of equivalence classes of functions in $\cM$ with respect to the equivalence relation defined by $a\sim b$ if $d^M(a,b)=0$, with the quotient metric $d$. 
\end{definition}

	\begin{lemma}
	If $F$ is an $n$-ary function symbol and $\bar a\in \cM^n$, then the interpretation $\omega\mapsto F^{M_\omega} (\bar a)$ is a measurable function. 

	If $R$ is an $n$-ary predicate symbol and $\bar a\in \cM^n$, then the interpretation $\omega\mapsto F^{M_\omega} (\bar a)$ is a measurable function. 
\end{lemma}
\begin{proof}
	We will prove the second assertion. Fix an $n$-ary predicate symbol~$R$ and an $n$-tuple $\bar a$ in $\cM$. Given $\varepsilon>0$, the syntax requires that there is  $\delta=\delta(\varepsilon)>0$ such that  for $n$-tuples $\bar x$ and~$\bar y$ in any $\calL$-structure $N$ we have that  $d(\bar a, \bar b)<\delta$ implies $|R^N(\bar x)-R^N(\bar y)|<\varepsilon$. 
	
	Since $e_n(\omega)$, for $n\in \bbN$, is dense in every $M_\omega$, there is a partition $\Omega=\bigsqcup_{i} X_i$ into measurable sets and a function $h\colon \bbN\times n\to \bbN$ such that for all $i$ and $j$ we have $d^{M_\omega}(a_j(\omega)-e_{h(i,j)}(\omega))<\delta$ for all $\omega\in X_i$. 
	Writing $\bar e(i)$ for the $n$-tuple in $\cM$ whose $j$-th coordinate agrees with $e_{h(i,j)}(\omega)$, the interpretation $\omega\mapsto R^{M_\omega}(\bar a)$ is uniformly  $\varepsilon$-approximated by 
	\[
	\omega\mapsto \sum_{i}\chi_{X_i}R^{M_\omega}(\bar e(i)). 
	\]
	By \eqref{2.Def.DirectIntegral} of Definition~\ref{Def.DirectIntegral},  this function is measurable. 
	Therefore the evaluation of $R$ at $\bar a$ is, as a uniform limit of measurable functions, measurable. 

The proof in case of function symbols is analogous. 
\end{proof}
	
	On $M$ the interpretations of function symbols and predicate symbols in~$\calL$ are defined in the natural way. The interpretation of a function symbol $F(\bar x)$ in $\calL$, for a tuple $\bar a$ in $M$ of appropriate sort, is the equivalence class $F^M(\bar a)$ of the function on $\Omega$ such that 
 \[
 \omega\mapsto F^{M_\omega}(\bar a_\omega).
 \]
 If $R(\bar x)$ is a relation symbol in $\calL$ and $\bar a$ in $M$ is of appropriate sort, then
 \[
 R^M(\bar a)=\int_\Omega R^{M_\omega}(\bar a_\omega)\, d\mu(\omega).
 \]

\begin{lemma} If $\calL$ is a continuous language and $M_\omega$, for $\omega\in \Omega$, is a measurable field of $\calL$-structures, then $\int^\oplus M_\omega\, d\mu(\omega)$ is an $\calL$-structure. \qed 
\end{lemma}	

\begin{proof}
	It is straightforward to verify that $M$ is complete with respect to $d$ and that the interpretation of each function and predicate symbol in a direct integral is continuous with respect to $d$, with the modulus of continuity as required by the syntax in $\calL$. The conclusion follows. 
\end{proof}

\subsubsection*{A remark on randomizations} Keisler's randomizations of discrete structures \cite{keisler1999randomizing} as well as their metric analog (\cite{ben2009randomizations}, \cite[Definition~3.4]{yaacov2013theories}) are closely related to direct integrals of measurable fields of structures.
Unlike direct integrals, randomizations are presented in an expanded two-sorted language. 
Precise relation between randomizations (of both discrete structures and continuous structures whose theory has an additional property, that the space of quantifier-free $n$-types is a Bauer simplex for all $n\geq 1$) is discussed in detail in \cite[\S 21]{yaacov2024extremal}. 
  Note that the quantifier elimination results for randomizations (\cite[Theorem~3.6]{keisler1999randomizing}, \cite[Theorem~2.9]{ben2009randomizations}, and \cite[Corollary~3.33]{yaacov2013theories}) refer to the expanded language. It is not difficult to see that, for example, a nontrivial direct integral of a measurable field of II$_1$ factors does not admit quantifier elimination in the language of tracial von Neumann algebras. As a matter of fact, no type II$_1$ tracial von Neumann algebra admits quantifier elimination by  the main result of \cite{farah2023quantifier}. This result has been improved further in \cite{farah2023quantum} where it was shown that these theories are not even model complete. 

\subsection{Elementarity}\label{S.elementarity}
The model theory of tracial von Neumann algebras and \cstar-algebras were introduced in \cite{FaHaSh:Model2}. 
For every tuple $\bar x=\langle x_0,\dots, x_{n-1}\rangle$ of variables ($n\geq 0$, allowing for the empty tuple) one associates the algebra of formulas $\frFx$ with free variables included in $\bar x$.  
If $\varphi(\bar x)$ is a formula with free variables included in $\bar x$, $(N,\tau)$ is a tracial von Neumann algebra, then the interpretation $\varphi^N(\bar a)$ is defined for every tuple $\bar a$ in $N$ of the appropriate sort.  Sentences are formulas with no free variables. 

To every tracial von Neumann algebra $N$ one defines a seminorm $\|\cdot\|_N$ on $\frFx$ by 
\[
\|\varphi(\bar x)\|_N=\sup \varphi^{N^\cU}(\bar a). 
\]
Here $\bar a$ ranges over all $n$-tuples in the unit ball of $N$. 
(The standard definition of $\|\varphi(\bar x)\|$ takes supremum over all structures $M$ elementarily equivalent to~$N$ and all $n$-tuples in $M$ of the appropriate sort, but  the two seminorms coincide.)

\begin{definition}
Suppose that $M$ and $N$ are tracial von Neumann algebras. They are said to be \emph{elementarily equivalent}, $M\equiv N$, if every sentence $\varphi$ satisfies $\varphi^M=\varphi^N$. 

An \emph{elementary embedding} $\Psi\colon M\to N$ is an embedding such that $\varphi^M(\bar a)=\varphi^N(\Psi(\bar a))$ for every $\varphi(\bar x)$ and every $\bar a$ of the appropriate sort. 

If $M$ is a subalgebra of $N$ and the identity map is an elementary embedding, $M$ is called an \emph{elementary submodel} of $N$, in symbols $M\preceq N$. 
\end{definition}

The diagonal embedding of $M$ into its ultrapower $M^\cU$ is elementary (\L o\'s's Theorem). 
If $M\preceq N$ and $N\preceq P$, then $M\preceq P$. If $M\preceq P$, $N\preceq P$, and $M\subseteq N$, then $M\preceq N$. However, $M\preceq P$, $M\preceq N$, and $P\subseteq N$ does not in general imply $P\preceq N$. 

\section{Definability in probability measure algebras}\label{S.LMBA}
In the present section we start the proof of our Feferman--Vaught theorem stated and proven in \S\ref{S.FV}, by laying down some general  definability results following a request of one of the referees.  Let $\LMBA$ denote the language of probability measure algebras as in \cite{berenstein2023model}. In addition to Boolean operations, this language is equipped with a predicate for a probability measure and metric derived from it. Thus if $(\cB,\mu)$ is a measure algebra, the distance is given by $d_\mu(A,B)=\mu(A\Delta B)$, and the language includes (symbols for) the  Boolean operations. For definiteness, if $(\cB,\mu)$ is a measure algebra then on $\cB^n$ we consider the max distance, $d_\mu(\bar A, \bar B)=\max_{i<n} \mu(A_i\Delta B_i)$. 

The following simple fact is a warm up for Lemma~\ref{L.definability} used in proof of Theorem~\ref{T.FV-integral}. 

\begin{lemma}  \label{L.definable.0} Each of the following sets is definable in every measure algebra $(\cB,\mu)$. 
	\begin{enumerate}
		\item \label{1.definable} The set $\cX_1=\{(A_1,A_2)\mid A_1\subseteq A_2\}$. 
		\item \label{2.definable} The set $\cX_2=\{(A,B,C)\mid A\cap B=C\}$. 
		\item  \label{3.definable} For  $A\in\cB$, the set $\cX_A=\{B\in \cB\mid B\subseteq A\}$ is definable with $A$ as a parameter. 
	\end{enumerate}
\end{lemma}

\begin{proof} In each of the instances, we need a formula that bounds the distance of an element (or a tuple of the appropriate sort) to the set in question (see \cite[\S 9]{BYBHU} and more specifically \cite[Definition 9.16]{BYBHU}).
	
	\eqref{1.definable} Let $\varphi(X_1,X_2)=\mu(X_1\setminus X_2)$. Clearly $\cX_1$ is the zero set of $\varphi$. 
	Also,  for every $\bar A\in \cB^2$  the pair $\bar B=(A_1\cap A_2, A_2)$  is in $\cX_1$ and $d_\mu(\bar A, \bar B)=\mu(A_1\Delta (A_1\cap A_2))=\varphi(A_1,A_2)$. 
	
	\eqref{2.definable} Let $\psi(X_1,X_2,X_3)=\mu((X_1\cap X_2)\Delta X_3)$. Clearly $\cX_2$ is the zero set of $\psi$ is $\cX_2$. Also, for every $\bar A\in \cB^3$ the triple $\bar B=(A_1,A_2, A_1\cap A_2)$ is in $\cX_2$ and $d_\mu(\bar A, \bar B)=\psi(\bar A)$, as required.  
	
	\eqref{3.definable} follows from \eqref{1.definable}. 
\end{proof}

Lemma~\ref{L.definability} and Lemma~\ref{L.definability+} below would be consequences of Lemma~\ref{L.definable.0} if it only were the case that an intersection of definable sets is definable. (Counterexamples can be found in as of yet unpublished papers  \cite{benyaacov2021continuous}  and \cite{farah3024tracial}.) The notation in these two lemmas is chosen to comply with the natural notation in the proof of Theorem~\ref{T.FV-integral} at the point when Lemma~\ref{L.definability+} is being invoked.

\begin{lemma}\label{L.definability} Suppose that $\ell\geq 1$ and $\bar U=(U_j)_{j<\ell}$ is a tuple in a measure algebra $(\cB,\mu)$ such that $U_0\geq U_1\geq \dots \geq U_{\ell-1}$. Then the set
	\[
	\cX[\bar U]=\{\bar Y\in \cB^\ell\mid Y_j\leq U_j\cap \bigcap_{i<j} Y_i \text { for }j<\ell\}
	\]
	is a definable set with parameter $\bar U$. 
\end{lemma}

\begin{proof}  As in Lemma~\ref{L.definable.0}, it suffices to find a formula $\varphi_{\bar U}$ such that its zero set is $\cX[\bar U]$ and for every $\bar X=(X_i: i< \ell )$ in $\cB^{\ell}$ the distance from $\bar X$ to $\cX[\bar U]$ is at most $\varphi_{\bar U}(\bar X)$.
	For an $\ell$-tuple $\bar U$ that satisfies $U_0\geq U_1\geq \dots \geq U_{\ell-1}$ let 
	\[
\textstyle	\varphi_{\bar U}(\bar X)=\max_{1\leq m<\ell} (\mu(X_m\setminus \bigcap_{j<m} X_j)+\mu(X_m\setminus U_m)). 
	\]
	Clearly $\cX[\bar U]$ is the zero set of $\varphi_{\bar U}$. Fix an $\ell$-tuple $\bar X$ in $\cB^\ell$. 
Let  
\[
Y_m=\bigcap_{j\leq m} X_j\cap U_m\text{ for }1\leq m<\ell.
\] 
Then $Y_m\subseteq \bigcap_{j< m} Y_j \cap U_m$ for all $1\leq m<\ell$, hence $\bar Y\in \cX[\bar U]$. 

In order to estimate $d_\mu(\bar X,\bar Y)$, note that since $Y_m\Delta X_m\subseteq (X_m\setminus \bigcap_{j<m} X_j)\cup (X_m\setminus U_m)$, we also have $d_\mu(X_m,Y_m)\leq \varphi_{\bar U}(\bar X)$, as required. 
\end{proof}

A little bit of natural (albeit slightly cumbersome) notation will be helpful. Suppose $\bbF$ is a finite set, $\ell(\zeta)\geq 1$ for $\zeta\in \bbF$, and we have 
\[
\bar U=(U^\zeta_{j})_{\zeta\in \bbF, j<\ell(\zeta)}.
\] 
Then we write $\bar U^\zeta=(U^\zeta_j)_{j<\ell(\zeta)}$ for $\zeta\in \bbF$. 

As pointed out above, an intersection of definable sets is not necessarily definable. However, in the following lemma we are dealing with a product of definable sets, not an intersection. 

\begin{lemma} \label{L.definability+} Suppose that $\bbF$ is a finite set and $\bar U^\zeta=(U^\zeta_j)_{j<\ell(\zeta)}$, for $\zeta\in \bbF$ and $\ell(\zeta)\geq 1$, is a tuple in a measure algebra $(\cB,\mu)$ such that $U^\zeta_0\geq U^\zeta_1\geq \dots \geq U^\zeta_{\ell(\zeta)-1}$. Then the set 
		\[
		\cY[\bar U]=\{Y^\zeta_j\mid \zeta\in \bbF, j<\ell(\zeta), \text{ and }(Y^\zeta_j)_{j<\ell(\zeta)}\in \cX[\bar U^\zeta]\}
		\]
		 is definable with parameter $\bar U=(\bar U^\zeta)_{\zeta\in \bbF}$. 
\end{lemma} 

\begin{proof} With $ \varphi_{\bar U}$ as in the proof of Lemma \ref{L.definability},  let 
	$
	\psi_{\bar U}(\bar X)=\max_{\zeta \in \bbF} \varphi_{\bar U^\zeta}(\bar X^\zeta)$.
	 	As there is no interaction between $\bar Y^\zeta$ for different $\zeta\in \bbF$,  and as 
	$\cY[\bar U]$ is equal to $\{\bar Y\mid \bar Y^\zeta\in \cX[\bar U^\zeta]$ for all $\zeta\in \bbF\}$, $\psi_{\bar U}$ witnesses that $\cY[\bar U]$ is definable, as required. 
\end{proof}
\section{The Feferman--Vaught type theorem for direct integrals} 
\label{S.FV}

Throughout this section we fix an arbitrary metric language $\calL$ and let $\LMBA$ be the language of probability measure algebras studied in \S\ref{S.LMBA}.

\begin{definition}
An $\LMBA$-formula $G(\bar X)$ in $m$ variables $\bar X=\langle X_1,\dots X_{m}\rangle$ is  \emph{coordinatewise increasing} if  for every measure algebra $(\cB,\mu)$ and every pair of $m$-tuples $\bar A=\langle A_i\rangle$ and $\bar A’=\langle A_i’\rangle$ in it, if  $A_i\leq A_i'$ for all $i\leq m$ then   $G^{(\cB,\mu)}(\bar A)\leq G^{(\cB,\mu)}(\bar A')$.    
\end{definition}

Definition~\ref{Def.1/k.determined} and Theorem~\ref{T.FV-integral} are stated for 
$\calL$-formulas whose ranges are included in $[0,1]$. 
Since the range of every $\calL$-formula $\varphi(\bar x)$ is a bounded interval,  the range of $r(\varphi(\bar x)-t)$ is $[0,1]$ for appropriately chosen real numbers~$r$ and $t$, and    this assumption will not result in loss of applicability of the theorem.   
In particular, the conclusion of Theorem~\ref{T.FV-integral} holds for tracial von Neumann algebras. 

	Given a 
 probability space $(\Omega,\cB,\mu)$ and a measurable field $M_\omega$, for $\omega\in \Omega$, of   $\calL$-structures, for an $\calL$-formula $\varphi\in\frFx$, $\bar a$ of the appropriate sort, and $t\in [0,1]$ we define
\begin{equation}\label{eq.Zeta}
		Z^{\zeta}_{t}[\bar a]=\{\omega\in \Omega: \zeta(\bar a_\omega)^{M_\omega}> t\}.
	\end{equation}

\begin{definition} \label{Def.1/k.determined}
	An $\calL$-formula $\varphi(\bar x)$ whose range is included in $[0,1]$ is \emph{determined in direct integrals of $\calL$-structures}, 
	if the following objects exist. 
	\begin{enumerate}[(D.1), widest=D.44.]
		\item \label{I.D.1} A finite set $\bbF[\varphi]$ of $\calL$-formulas whose free variables are included in the free variables of $\varphi(\bar x)$ and whose ranges are included in $[0,1]$. 
		\item \label{I.D.2} For every $k\geq 2$, an integer $l(k,\varphi,\zeta)\geq 1$ and a coordinatewise increasing $\LMBA$-formula $G_{\varphi,k}(\bar X)$ with $\sum_{\zeta\in \bbF[\varphi]}l(k,\varphi,\zeta) $ many variables $X^\zeta_{i/l(k,\varphi,\zeta)}$, for $\zeta\in \bbF[\varphi]$ and $i<l(k,\varphi,\zeta)$. 
			\pushcounter
	\end{enumerate}
These objects are required to be such that  for every 
 probability space $(\Omega,\cB,\mu)$ and a measurable field $M_\omega$, for $\omega\in \Omega$, of   $\calL$-structures, and $\bar a$ of the appropriate sort the following hold. (Writing $M=\int^\oplus_\Omega M_\omega\, d\mu(\omega)$.) 
	\begin{enumerate}[(D.1), widest=D.44.]
		\popcounter
		\item\label{I.theta.l.1} $\varphi(\bar a)^M> t/k$ implies  
		\[
		 G_{\varphi,k}(Z^{\zeta}_{i/l(k,\varphi,\zeta)}[\bar a]; i<l(k,\varphi,\zeta), \zeta\in \bbF[\varphi])>(t-1)/k. 
		\]
		\item\label{I.theta.l.2}  $ G_{\varphi,k}(Z^{\zeta}_{i/l(k,\varphi,\zeta)}[\bar a]; i<l(k,\varphi,\zeta), \zeta\in \bbF[\varphi])>t/k$
		implies 
		\[
		\varphi(\bar a)^M> (t-1)/k. 
		\]
	\end{enumerate}
Similarly, if \ref{I.D.1} holds and \ref{I.D.2}—\ref{I.theta.l.2}  hold for a specific value of $k$, we say that $\varphi$ is \emph{$k$-determined}. 
\end{definition}
In particular, Definition~\ref{Def.1/k.determined} asserts that  the value of $\varphi(\bar a)$  is determined up to $2/k$  by the value of $G_{\varphi,k}$, which is in turn determined by the distribution of the evaluations of formulas $\zeta$ in the finite set $\bbF[\varphi]$, up to (roughly) $1/l(k,\varphi,\zeta)$ in the measurable field~$M_\omega$.

On the set of all $\calL$-formulas consider the natural uniform metric, 
\[
d(\varphi(\bar x),\psi(\bar x))=\sup|\varphi^M(\bar a)-\psi^M(\bar a)|,
\] 
where the supremum is taken over all $\calL$-structures $M$ and all tuples $\bar a$ of the appropriate sort in $M$.

\begin{theorem} \label{T.FV-integral} For every metric language $\calL$ the set of all determined formulas is dense in the set of all $\calL$-formulas. 
	\end{theorem}


\begin{proof}
	The proof proceeds by induction on complexity of $\varphi$, simultaneously for all~$k\geq 2$. It will be clear from the proof that the set $\bbF[\varphi]$ does not depend on the choice of $k$. This is essentially the set of all subformulas of $\varphi$. 
	
	By \cite[Proposition 6.6]{BYBHU}, any set of $\calL$-formulas that includes the atomic formulas and is closed under multiplication by $1/2$, the operation $\varphi\dminus\psi=\max(0,\varphi-\psi)$, and quantifiers $\inf$ and $\sup$ is dense in the set of all $\calL$-formulas. It therefore suffices to prove that
	the sets of all $k$-determined formulas satisfy the following closure properties:
	\begin{enumerate}
		\item \label{1.FV} All atomic formulas are $k$-determined. 
		\item \label{2.FV} If $\varphi$ is $k$-determined, so is $\frac 12 \varphi$. 
		\item \label{3.FV} If $\varphi$ and $\psi$ are $3k$-determined, then  $\varphi\dminus \psi$ is $k$-determined. 
		\item \label{4.FV} If $\varphi$ is $k$-determined, so are $\sup_x\varphi$ and $\inf_x\varphi$
		for every variable $x$. 
	\end{enumerate}
	For readability of the ongoing proof, presented  by induction on the complexity of $\varphi$ simultaneously for all $k\geq 2$, we combine the recursive construction of $\bbF[\varphi]$, $l(k,\varphi,\zeta)$ for $\zeta\in \bbF[\varphi]$,  and $G_{\varphi,k}$ with a proof that 
	these objects have the required properties 
	for an arbitrary probability space $(\Omega,\cB,\mu)$, a measurable family of $\calL$-structures  $(M_\omega)_{\omega\in\Omega}$, and its direct integral 
	\[
	M=\int^\oplus_\Omega M_\omega\, d\mu(\omega)
	\]
	(Needless to say, the constructed objects will not depend on the choice of the measure space or the measurable family.)

	
	\eqref{1.FV}
	Suppose that $\varphi(\bar x)$ is atomic or constant (i.e., a scalar) and that $\bar a\in M$ is of the appropriate sort. Let $\bbF[\varphi]=\{\varphi\}$,  $l(k, \varphi)=k$,  and define 
	$$
	G_{\varphi,k}(\bar X) = \frac 1k \sum_{i=1}^{k-1}\mu(X_i).
	$$
	It is clear that this formula is coordinatewise increasing. Then for $0\leq i<k$,  with
	\[
	Z^{\varphi}_{i/k}[\bar a]=\{\omega\in \Omega\mid \varphi(\bar a_\omega)^{M_\omega}> i/k\}
	\]
	 we have (for simplicity we will write $\bar Z^\varphi= ( Z^\varphi_0[\bar a], \dots,  Z^\varphi_{(k-1)/k}[\bar a])$). From the layer cake decomposition formula for the integral of a nonnegative function we obtain the following. 
	$$
	G_{\varphi,k}(\bar Z^\varphi)\leq
	\int_\Omega\varphi(\bar a_\omega)^{M_\omega}\, d\mu(\omega)
	\leq\frac 1k \sum_{i=0}^{k-1}\mu( Z^\varphi_{i/k}[\bar a]) \leq G_{\varphi,k}(\bar Z^\varphi)+\frac 1k.
	$$
	Thus the conditions \ref{I.theta.l.1} and \ref{I.theta.l.2} are clearly satisfied.
	
		\eqref{2.FV}
	Suppose that $\varphi(\bar x)=\frac 12 \psi(\bar x)$ and~$\psi$ is $k$-determined. 
	Let $\bbF[\varphi]=\bbF[\psi]$, $l(k,\varphi,\zeta)=l(k,\psi,\zeta)$ (we could have taken $l(k,\varphi,\zeta)$ to be $\lceil l(k,\psi,\zeta)/2\rceil$, but there is no reason to be frugal)   and define
	$$G_{\varphi,k} (\bar X)=\frac 12 G_{\psi,k}(\bar X).$$ 
	These objects satisfy the requirements by the definitions. 
	
	\eqref{3.FV} Suppose that $\varphi=\psi\dminus \eta$ and each one of $\psi$ and~$\eta$ is $3k$-determined. In order to prove that $\varphi$ is $k$-determined let 
	\[
	\bbF[\varphi]=\bbF[\psi]\cup \{1-\zeta: \zeta\in \bbF[\eta]\}.
	\]
 (If $\zeta\in\bbF[\eta]$ then its range is included in $[0,1]$, hence the range of $1-\zeta$ is also included in $[0,1]$.) 
Also, let   $l(k,\varphi,\zeta)=l(3k,\psi,\zeta)$ for $\zeta\in \bbF[\psi]$, and  $l(k,\varphi,1-\zeta)=l(3k,\eta,\zeta)$ for $\zeta\in \bbF[\eta]$. 
To define $G_{\varphi,k}$, we need an additional bit of notation. 

For $\zeta\in \bbF[\eta]$ and 
$s\in[0,1]$  let  (writing $\ell=l(3k, \eta, \zeta)$ for readability)
\[
\widetilde Z^{1-\zeta}_s=\{\omega\in \Omega :~ 1-\zeta(\bar a_\omega)^M \geq s\}.
\]
For a tuple $\bar Z^{1-\zeta}=(Z^{1-\zeta}_{0}, \dots , Z^{1-\zeta}_{\frac{\ell-1}{\ell}})$ let 
\[
\overleftarrow Z^{1-\zeta} = (\widetilde Z^{1-\zeta}_{\frac{\ell-1}{\ell}}, \dots, \widetilde Z^{1-\zeta}_0). 
\] 

First, note that for every $0\leq i\leq \ell$ and $\zeta\in \bbF[\eta]$ we have  ($Y^\complement$ denotes the complement of $Y$, applied pointwise if $Y$ is a tuple)
\begin{align*}
    \Big(\widetilde{Z}_{(\ell-i)/\ell}^{1-\zeta}[\bar a]\Big)^\complement& = \big\{\omega \in \Omega : ~ \big(1-\zeta(\bar a_\omega)\big)^{M_\omega} \geq\frac{\ell-i}{\ell} \big\}^\complement\\
     &= \big\{\omega \in \Omega : ~ \zeta(\bar a_\omega)^{M_\omega} \leq \frac{i}{\ell} \big\}^\complement\\
     &= \{\omega\in \Omega :~ \zeta(\bar a_\omega)^{M_\omega} > \frac{i}{\ell}\} = Z_{i/\ell}^{\zeta}[\bar a].
\end{align*}
This shows that $(\overleftarrow {Z^{1-\zeta}})^\complement = \bar Z^\zeta$. Define
\footnote{For simplicity, we present $G_{\varphi,k}$ not as a formula in a tuple $\bar X$ abstract variables, but in terms of the intended values for these variables.  Note that, since $(\overleftarrow Z^{1-\zeta})^\complement$ consists of the complements of sets in  $\bar Z^{1-\zeta}$, $G_{\varphi,k}$ depends on the correct choice of variables, $\bar Z^\zeta$ for $\zeta\in \bbF[\varphi]$.}
 \begin{multline*}
 G_{\varphi,k}\big(\bar Z^{\xi}, \bar Z^{1-\zeta}; \xi\in \bbF[\psi], \zeta\in \bbF[\eta] \big)\\
 =  G_{\psi,3k}(\bar Z^{\xi}, \xi\in \bbF[\psi])\dminus  G_{\eta,3k}({(\overleftarrow Z^{1-\zeta}})^\complement, \zeta\in \bbF[\eta] ).
 \end{multline*}
Then $G_{\varphi,k}$ is coordinatewise increasing since the same is true for $G_{\psi,3k}$	and~$G_{\eta,3k}$.
	
        \begin{claim}\label{claim dminus} The formula
		$G_{\varphi,k}$ satisfies \ref{I.theta.l.1} and \ref{I.theta.l.2} of Definition~\ref{Def.1/k.determined}.
		\end{claim}
        \begin{proof}
		Suppose $\varphi^{M}(\bar a)>t/k$, for $M=\int^\oplus_\Omega M_\omega \, d\mu(\omega)$ and $\bar a$ in~$M$ of the appropriate sort. There exists $m<3(k- t)$ such that  
		\begin{itemize}
			\item  $\psi(\bar a)^M>  (3t+m)/3k$, and 
			\item $\eta(\bar a)^M \leq (m+1)/3k$. 
		\end{itemize}
	By the induction hypothesis,
	\begin{itemize}
			\item[(IH1)]  $G_{\psi, 3k}(\bar Z^{\xi}; ~\xi \in\bbF[\psi])>  \frac{3t+m-1}{3k}$, and 
			\item[(IH2)] $G_{\eta, 3k}(\bar Z^{\zeta}; ~\zeta  \in\bbF[\eta]) \leq \frac{m+2}{3k}$. 
		\end{itemize}

By (IH1) and (IH2), we have
\begin{align*}
    G_{\varphi,k}\big(\bar Z^{\xi}, \bar Z^{1-\zeta}; &\xi\in \bbF[\psi], \zeta\in \bbF[\eta] \big)\\&= G_{\psi,3k}(\bar Z^{\xi};\xi\in \bbF[\psi])\dminus  G_{\eta,3k}((\overleftarrow{ Z^{1-\xi}})^\complement;\zeta\in\bbF[\eta]) \\
    & \geq G_{\psi,3k}(\bar Z^{\xi};\xi\in \bbF[\psi])\dminus  G_{\eta,3k}(\bar Z^\zeta;\zeta\in\bbF[\eta]) \\
    &> \frac{3t+m-1}{3k}  - \frac{m+2}{3k} = \frac{t-1}{k}.
\end{align*}
This completes the proof of \ref{I.theta.l.1}.

To prove \ref{I.theta.l.2}, suppose that $G_{\varphi,k}\big(\bar Z^{\xi}, \bar Z^{1-\zeta}; \xi\in \bbF[\psi], \zeta\in \bbF[\eta] \big)>t/k$. Then by the definition of  $G_{\varphi,k}$, for some $m<3(k-t)$ we have 
  \begin{itemize}
      \item $G_{\psi,3k}(\bar Z^{\xi};\xi\in \bbF[\psi])>\frac{3t+m}{3k}$, and 
      \item $G_{\eta,3k}(\bar Z^\zeta;\zeta\in\bbF[\eta])\leq \frac{m+1}{3k}$.
  \end{itemize}
  By the induction hypothesis this implies 
	\begin{itemize}
			\item[(IH3)]  $\psi(\bar a)^M>  \frac{3t+m-1}{3k}$, and  
			\item[(IH4)] $\eta(\bar a)^M \leq \frac{m+2}{3k}$. 
		\end{itemize}
The conditions (IH3) and (IH4) immediately imply that $\varphi(\bar a)^M > (t-1)/k$ and complete the proof. 
\end{proof}

	\eqref{4.FV}  Suppose that $\varphi(\bar x)=\sup_y \psi(\bar x, y)$ and $\psi$ is $k$-determined. Let 
\begin{align*}
	\calR[\zeta]&=\{i/l(k,\psi,\zeta)\mid i<l(k,\psi,\zeta)\},\quad \text{for }\zeta\in \bbF[\psi],\\
	\cC&=\{\alpha\mid \text{there is a nonempty  $\bbF\subseteq\bbF[\psi]$ such that $\alpha\in \textstyle\prod_{\zeta\in \bbF} \calR[\zeta]$}\}.
	\end{align*}
	One may think of $\alpha\in \cC$ as a function from $\bbF$ into $\bbQ\cap [0,1]$. The point of specifying $\alpha(\zeta)\in \calR[\zeta]$ is that, because each $\calR[\zeta]$ is finite, the set 
	$\cC$ is finite as well.
	
	For $\alpha\in \cC$ define the $\calL$-formula
	\begin{align}\label{eq.xi-alpha}
	\xi_{\alpha}(\bar x)&=\sup_y \min_{\zeta\in \dom(\alpha)} (\zeta(\bar x,y)-\alpha(\zeta)).
	\end{align}
	Then for every  $\alpha\in\cC$, and $\bar a$ in $M$  we have 
	\begin{align*}\label{Z-xi-alpha}
	Z^{\xi_\alpha}_{0}[\bar a] =& \{\omega\mid \xi_\alpha(\bar a_\omega)^{M_\omega}>  0\} \\ 
 = & \{\omega\mid \textstyle\sup_{y\in M_\omega}\textstyle\min_{\zeta\in \dom(\alpha)} ((\zeta(\bar a,y) - \alpha(\zeta)) >  0\} \\
\subseteq & \bigcap_{\zeta\in\dom(\alpha)}\{\omega\mid \textstyle\sup_{y\in M_\omega} \zeta(\bar a,y)  > \alpha(\zeta)\} \\
 = &  \bigcap_{\zeta\in\dom(\alpha)} Z^{\xi_\zeta}_ {\alpha(\zeta)}[\bar a]
	\end{align*}
	Let 
	\begin{align*}
	\bbF[\varphi]&=\{\xi_\alpha\mid \alpha\in \cC\},\\
	l(k,\xi_\alpha,\varphi)&=\max\{l(k,\zeta,\psi)\mid \zeta\in \dom(\alpha)\},\quad\text{for $k\geq 2$ and  }\alpha\in \cC. 
	\end{align*}
	For simplicity  of notation we will denote $Z^{\xi_\alpha}_{0}[\bar a] $ with $Z^{\xi_\alpha}_{0}$ and, more generally,   $Z^\eta_r[\bar a]$ with $Z^\eta_r$, whenever there is no ambiguity. It will also be helpful to introduce an abbreviation and write, for $\zeta\in \bbF[\psi]$, 
	\[
	\ell(\zeta)=l(k,\psi,\zeta). 
	\]

Prior to defining the $\LMBA$-formula $G_{\varphi,k}$, we note that  every variable occurring in $G_{\psi,k}$  is of the form~$Z^\zeta_{i/\ell(\zeta)}$ for some $\zeta\in \bbF[\psi]$ and $i< \ell(\zeta)$. Let 
	\begin{equation}\label{eq.barZ}
	\bar Z=(Z^{\xi_\alpha}_0 \mid \alpha\in \cC ).
	\end{equation}
	
\begin{claim}\label{definability claim} With the notation from the previous paragraph, 
	 the set $\cY[\bar Z]$ of all $\bar Y=(Y^\zeta_i \mid i<\ell(\zeta), ~ \zeta\in \bbF[\psi])$ that satisfy conditions 
	\begin{enumerate}[($\theta$.1), widest=$\theta$.4.]
		\item \label{I.FV.1} $Y^\zeta_i\subseteq Z^{\xi_{\alpha}}_0$ for all $\alpha\in \cC$ such that $\zeta\in \dom(\alpha)$ and $\alpha(\zeta)=\frac{i}{\ell(\zeta)}$.  
		\item \label{I.FV.2} $Y^\zeta_i \subseteq  Y^\zeta_{i-1}$ if $i\geq 1$.  
		\pushcounter
	\end{enumerate} 
	is definable with parameters $\bar Z$ as in \eqref{eq.barZ}.
\end{claim}	
\begin{proof}
	 For  $m<\ell(\zeta)$, let 
	\[
	U^\zeta_m =\bigcap\{Z^{\xi_\alpha}_0\mid \alpha\in \cC, \zeta\in \dom(\alpha), \alpha(\zeta)\leq \frac{m}{\ell(\zeta)}\}. 
	\]
	Then $U^\zeta_0\supseteq U^\zeta_1\supseteq \dots \supseteq U^\zeta_{\ell(\zeta)-1}$ for every $\zeta$. 
Let 		
\[
		\cX[\bar U^\zeta]=\{\bar Y^\zeta\in \cB^{\ell(\zeta)}\mid Y^\zeta_j\leq U^\zeta_j\cap \bigcap_{i<j} Y^\zeta_i\text { for }j<\ell(\zeta)\}. 
		\] 
	Then $\cY[\bar U]=\{Y^\zeta_j\mid \zeta\in \bbF[\psi], j<\ell(\zeta), \text{ and }(Y^\zeta_j)_{j<\ell(\zeta)}\in \cX[\bar U^\zeta]\}$, as considered in Lemma~\ref{L.definability+} is definable. This set is equal to $\cY[\bar Z]$ and it is definable with parameters $\bar Z$. 
\end{proof}
	
Therefore $G_{\varphi,k}$ as defined below is a formula (on the right-hand side, in $G_{\psi,k}$ the variable $Z^\zeta_{i/l(\zeta)}$	is replaced with $Y^\zeta_i$ for all $\zeta\in \bbF[\psi]$ and $i<\ell(\zeta)$):
 \begin{equation*}
  G_{\varphi,k}(\bar Z^{\xi}\colon  \xi\in \bbF[\varphi])
  =\sup_{\bar Y\in \cY[\bar Z]} G_{\psi,k}(Y^\zeta_i; \zeta\in\bbF[\psi], i<l(\zeta)). 
 \end{equation*}
It is clear that $G_{\varphi,k}$ is coordinatewise increasing since $G_{\psi,k}$ has this property and since the set $\cY[\bar Z]$ is also increasing  in $\bar Z$ (in the sense that $\bar Z\leq \bar Z’$ implies $\cY[\bar Z]\subseteq \cY[\bar Z’]$). 
It remains to prove that $G_{\varphi,k}$ satisfies the requirements of Definition~\ref{Def.1/k.determined}. 

To prove \ref{I.theta.l.1},
suppose $\varphi(\bar a)^M>t/k$, for $M=\int^\oplus_\Omega M_\omega \, d\mu(\omega)$ and $\bar a$ in $M$ of the appropriate sort. Pick $b\in M$ such that $\psi(\bar a, b)^M>t/k$. Then, by the induction hypothesis
\begin{equation}\label{Gpsi}
G_{\psi,k}\big( \bar Z^\zeta [\bar a, b]; \zeta \in \bbF[\psi]\big)>\frac{t-1}{\ell(k)}. 
\end{equation}
 For $\zeta\in \bbF[\psi]$ and $i<\ell(\zeta)$ let
\[
Y^\zeta_i =Z^{\zeta}_{i/l(\zeta)}[\bar a,b].   
\]
We claim that $\bar Y$ defined in this manner belongs to the set $\cY[\bar Z]$ as in Claim~\ref{definability claim}. 
Condition \ref{I.FV.2} is clearly satisfied and condition \ref{I.FV.1}  is  satisfied because for all $\zeta\in \bbF[\psi]$ and $i<l(\zeta)$, we have that
\begin{align*}
	Y^\zeta_i = &\{\omega\in \Omega\} \mid \zeta(\bar a_\omega, b_\omega)^{M_\omega}> i/\ell(\zeta) \}\\ \subseteq  & \bigcap_{\stackrel{\alpha\in \cC,  \zeta\in \dom(\alpha),}{\alpha(\zeta)\leq i/\ell(\zeta)}}\{ \omega\in \Omega \mid  \quad \sup_{y\in M_\omega} \min_{\zeta\in \dom(\alpha)} \eta(\bar a_\omega,y)^{M_\omega}- \alpha(\zeta) > 0\}\\
=&  \bigcap_{\stackrel{\alpha\in \cC,  \zeta\in \dom(\alpha),}{\alpha(\zeta)\leq i/\ell(\zeta)}} Z^{\xi_\alpha}_0.
\end{align*}
By  \eqref{Gpsi} we have
\[
G_{\varphi, k }\big(\bar Z^{\xi} ; \xi\in \bbF[\varphi] \big) \geq G_{\psi,k}(Y^\zeta_i; \zeta\in\bbF[\psi], i<\ell(\zeta))>\frac{t-1}{\ell(k)}.
\]
This completes the proof of \ref{I.theta.l.1}.

To prove \ref{I.theta.l.2},  assume $G_{\varphi, k }\big(\bar Z^{\xi_\zeta}[\bar a] ;  \zeta \in \bbF[\psi] \big)>t/k$. Then there are measurable  sets  (as before, $\bar a$ is suppressed for readability) $Y^\zeta_i$ for $\zeta\in \bbF[\psi]$ and $i<\ell(\zeta)$  satisfying \ref{I.FV.1} and \ref{I.FV.2} such that
\begin{equation}\label{G-Y}
 G_{\psi,k}(Y^\zeta_i; \zeta\in\bbF[\psi], i<\ell(\zeta)) >\frac{t}{k}.
\end{equation}

For each $\omega\in \Omega$ let 
\[
D_\omega = \{\zeta\in \bbF[\psi]: \omega\in Y^\zeta_0\}.
\]

Define $\alpha_\omega\in \cC$ with $\dom(\alpha_\omega)=D_\omega$ by
\begin{equation}\label{alpha_gamma}
  \alpha_\omega(\zeta) =\max\{i/\ell(\zeta) : \omega\in Y^\zeta_i\}.
\end{equation}
If $D_\omega\neq \emptyset$ then $\omega\in \bigcap_{\zeta\in D_\omega} Y^\zeta_{\alpha_\omega(\zeta)}$ and  \ref{I.FV.1}   implies $Y^\zeta_{\alpha_\omega(\zeta)}\subseteq Z^{\xi_{\alpha_\omega}}_0$ for every $\zeta\in \dom(\alpha_\omega)$, hence $\bigcap_{\zeta\in D_\omega} Y^\zeta_{\alpha_\omega(\zeta)}\subseteq Z^{\xi_{{\alpha}_\omega}}_0$ and 
		$\omega\in Z^{\xi_{{\alpha}_\omega}}_0$. 
	Therefore, 
\[
\sup_y \min_{\zeta\in D_\omega} \zeta(\bar a_\omega, y)>\alpha_\omega(\zeta).
\]
Recall from Definition~\ref{Def.DirectIntegral} that all $M_\omega$, for $\omega\in \Omega$, are separable and that $e_n$, for $n\in \bbN$, enumerate a subset of $M$ such that $e_n(\omega)$, for $n\in \bbN$, form a dense subset of $M_\omega$ for every $\omega\in \Omega$. 
Thus, if $D_\omega\neq \emptyset$ there exists $n\in \bbN$ such that 
\[
\min_{\zeta\in D_\omega} \zeta^{M_\omega}(\bar a_\omega, e_{n}(\omega))>\alpha_\omega(\zeta).
\]
Let $n(\omega)$ be the minimal $n$ with this property. For each $n\in \bbN$, let
\[
\Omega_n=\{\omega\in \Omega\mid n(\omega)=n\}
\] 
and $\Omega_\infty= \{\omega\in \Omega\mid  D_\omega =\emptyset\}$. Note that $\Omega_\infty = \Omega\setminus \bigcup_{\zeta\in \bbF[\psi]} Y^\zeta_0$. Therefore, $\{\Omega_n \mid n\in \bbN\cup\{\infty\}\}$ is a partition of $\Omega$ into measurable sets. 
The function $b$ defined on $\Omega$ by
\[
b_\omega=\begin{cases} e_{n}(\omega) &\text{if } \omega\in \Omega_n\\ 
e_0(\omega) &\text{if } \omega\in \Omega_\infty
\end{cases}
\]
is a measurable field of elements and it therefore defines an element of $M$. By the choice of $n(\omega)$, if $D_\omega\neq \emptyset$, then
\begin{equation}\label{ineq-gamma}
\min_{\zeta\in D_\omega}\zeta(\bar a_\omega, b_\omega)^{M_\omega} > \alpha_\omega(\zeta).
\end{equation}

\begin{claim}\label{YsubZ}
We have     $Y^\zeta_i\subseteq Z^{\zeta}_{i/\ell(\zeta)}[\bar a, b]$ for all $\zeta\in\bbF[\psi]$ and $i<\ell(\zeta)$.
\end{claim}
\begin{proof}
Suppose $\omega\in Y^\zeta_i$.  Then, from \eqref{alpha_gamma} we have $i/ \ell(\zeta)\leq \alpha_\omega(\zeta)$. By the choice of $b$, 
\[
\zeta(\bar a_\omega, b_\omega)^{M_\omega} > \alpha_\omega(\zeta)\geq\frac{i}{\ell(\zeta)},
\]
which means, $\omega\in Z^{\zeta}_{i/\ell(\zeta)}[\bar a, b]$. 
\end{proof}

Since $G_{\psi,k}$ is a coordinatewise increasing formula,  Claim~\ref{YsubZ} and \eqref{G-Y} imply
\[
G_{\psi,k}\big( \bar Z^\zeta[\bar a, b]; \zeta \in \bbF[\psi]\big)>\frac{t}{k}. 
\]
The inductive hypothesis implies that $\psi(\bar a,b)^M>(t-1)/k$, and therefore   $\varphi(\bar a)^M>(t-1)/k$ as required. 

Since $\inf_y \psi(\bar x,y)=1-\sup_y (1-\psi(\bar x,y))$, the case when $\varphi(\bar x)=\inf_y \psi(\bar x,y)$ for some $\psi$ that satisfies the inductive assumption follows from the previous case.  This completes the proof by induction on complexity of $\varphi$. 
 \end{proof}

What makes Theorem \ref{T.FV-integral}  (or in general, Feferman-Vaught-type theorems) `effective' is that the objects in  \ref{I.D.1}—\ref{I.D.2} of  Definition~\ref{Def.1/k.determined} can be recursively obtained from only the syntax of $\varphi$, as the proof shows.  In order to state Corollary~\ref{FV-corollary 1} we need a definition. 

\begin{definition}[Measurable fields and direct integrals of metric structures] \label{Def.DirectIntegral.2}
	Suppose that $\calL$ is a continuous language, $(\Omega,\cB,\mu)$ is a separable probability measure space, $(N_\omega, d_\omega)$, for $\omega\in \Omega$, are separable $\calL$-structures and~$M_\omega$ is a substructure of $N_\omega$ for a set of $\omega$ of full measure. Assume that $e_n$, for $n\in \bbN$, is a sequence in $\prod_{\omega\in \Omega} N_\omega$ such that the following two conditions hold. 
	\begin{enumerate}
		\item For every $\omega$ the set $\{e_{2n}(\omega)\mid n\in \bbN\}$ is dense in $M_\omega$ and the set $\{e_{n}(\omega)\mid n\in \bbN\}$ is dense in $N_\omega$. 
		\item For every predicate $R(\bar x)$ in $\calL$ and every tuple $e_{2\bar n}=\langle e_{2n(0)}, \dots, e_{2n(j-1)}\rangle$ of the appropriate sort, the function $\omega\mapsto R^{M_\omega}(e_{2\bar n}(\omega))$ is measurable. 
		\item For every predicate $R(\bar x)$ in $\calL$ and every tuple $e_{\bar n}=\langle e_{n(0)}, \dots, e_{n(j-1)}\rangle$ of the appropriate sort, the function $\omega\mapsto R^{N_\omega}(e_{\bar n}(\omega))$ is measurable. 
	\end{enumerate}
As in Definition~\ref{Def.DirectIntegral},	the structures $N_\omega$, together with the functions $e_n$, form a \emph{measurable field} of $\calL$-structures and  the structures $M_\omega$ form a \emph{measurable subfield} of this measurable field. 
	\end{definition}

\begin{corollary}\label{FV-corollary 1}
    Suppose $(\Omega, \cB,\mu)$ is a separable measure space, and $M_\omega$ and $N_\omega$ are measurable fields of  structures of the same language, for all $\omega\in \Omega$.

    If $M_\omega\equiv N_\omega$ for almost all $\omega$, then  
    \[
    \int_\Omega^\oplus M_\omega \, d\mu(\omega)  \equiv \int_\Omega^\oplus N_\omega \, d\mu(\omega).
    \]
    If $M_\omega$, for $\omega\in \Omega$ is a measurable subfield of $N_\omega$, for $\omega\in \Omega$ and in addition $M_\omega\preceq N_\omega$ for almost all $\omega$, then 
\[
\int_\Omega^\oplus M_\omega \, d\mu(\omega)  \preceq \int_\Omega^\oplus N_\omega \, d\mu(\omega).
\]
\end{corollary}
\begin{proof}
We prove the second part. Fix a formula $\varphi(\bar x)$ and $k\geq 2$.  By Theorem \ref{T.FV-integral} can be uniformly approximated by a formula that is determined. Therefore, without loss of generality we assume $\varphi(\bar x)$ is determined.   Let 
\[
M= \int_\Omega^\oplus M_\omega \, d\mu(\omega)  ~~\text{ and } ~~ N= \int_\Omega^\oplus N_\omega \, d\mu(\omega).
\]
 For  every $\bar a$ in $M$ of the appropriate sort and every formula $\zeta(\bar x)\in \bbF[\varphi]$, the set of $\omega$ such that 
$\zeta(\bar a_\omega)^{N_\omega}=\zeta(\bar a_\omega)^{M_\omega}$ has full measure. That is, the sets of the form
$Z^\zeta_r [\bar a]$, as in Definition \ref{Def.1/k.determined}, evaluated in structures $M$ and $N$ are the same.    
 Therefore $|\varphi(\bar a)^N- \varphi(\bar a)^M|<2/k$ and because $k$ was arbitrary it follows that $\varphi(\bar a)^N=\varphi(\bar a)^M$. Since $\varphi$ and $\bar a$ were arbitrary, $N_\omega\preceq M_\omega$. 

Proof of the first part is analogous. 
\end{proof}

As Ita\" i Ben Ya'acov pointed out, Corollary~\ref{FV-corollary 1} can be proven using quantifier elimination in atomless randomizations, as proven in \cite{ben2009randomizations}. This result applies only to atomless measure spaces but is in this case even slightly stronger as it shows that the direct integrals are elementarily equivalent even as randomization structures. 

A special case of Corollary \ref{FV-corollary 1} where all the fiber of the direct integrals are the same tracial von Neumann algebra leads to the following corollary.
\begin{corollary}\label{FV-corollary 2}
    Suppose $M$ and $N$ are elementarily equivalent tracial von Neumann algebras and $(\Omega, \cB,\mu)$ is a separable measure space. Then 
    \[
    M\bar\otimes L^\infty(\Omega, \mu)\equiv N\bar\otimes L^\infty(\Omega, \mu).
    \]
    If $M\preceq N$ then 
    \[
M\bar\otimes L^\infty(\Omega, \mu)\preceq N\bar\otimes L^\infty(\Omega, \mu).
\]\end{corollary}

\section{Applications to tracial von Neumann algebras}\label{S.vNA}

In this section we prove Theorem~\ref{T.tensor}, after discussing a technical point. 

 \subsection{Two languages for tracial von Neumann algebra} Tracial von Neumann algebras are equipped with a distinguished tracial state $\tau$, usually suppressed for the simplicity of notation.\footnote{For simplicity, in this proof we allow tracial von Neumann algebras whose distinguished trace is not normalized.} In the literature tracial von Neumann algebras are usually considered with respect to the $\|\cdot\|_2$-norm, 
 \[
 \|a\|_2=\tau(a^*a)^{1/2},
 \]  
 but in \cite[\S 29]{yaacov2024extremal} they are for convenience considered with respect to the $\|\cdot\|_1$ norm, 
\[
\|a\|_1=\tau((a^*a)^{1/2}). 
\]
We will denote the corresponding languages $\LLTWO$ and  $\LLONE$, respectively. Since the syntax of continuous logic requires each function symbol to be equipped with a modulus of uniform continuity, the difference between these two languages is not only notational. 
By \cite[Lemma~29.1]{yaacov2024extremal}, on operator norm-bounded  balls  the $\|\cdot\|_1$-norm is compatible with the strong operator topology (and therefore equivalent to the $\|\cdot\|_2$-norm). 
We thus have two competing languages and two competing axiomatizations (the standard one and the one in \cite[Proposition~29.4]{yaacov2024extremal}) of tracial von Neumann algebras in continuous logic. In order to facilitate the ongoing discussion, for $j=1,2$ we will refer to the axiomatization (formulas,  definable predicates, etc.) using the $\|\cdot\|_j$-norm as 
$j$-axiomatization ($j$-formulas, $j$-definable predicates, etc.). 

\begin{lemma}\label{L.affine.12}
	\begin{enumerate}
		\item 	The tracial state is a j-definable predicate for $J=1,2$. 
		\item The norm $\|\cdot\|_2$ is a 1-definable predicate. 
		\item The norm $\|\cdot\|_1$ is a 2-definable affine predicate. 
		\item Every 2-definable predicates is a 1-definable predicate and vice versa.  
	\end{enumerate}
\end{lemma}

\begin{proof}	We prove that the tracial state is a 1-definable  predicate by exhibiting a concrete defining formula.
	If $a=a^*$ and $\|a\|\leq n$ (data visible from the sort of $a$) then $|a+n|=a+n$ and $\tau(a)=\|a+n\|_1-n$. Since $a$ can be written as $a=a_0+i a_1$ where $a_0:=(a+a^*)/2$ and $a_1:=(a-a^*)/2i$ are self-adjoint, we have that (still assuming $\|a\|\leq n$) 
	\[
	\tau(a)=\|a_0+n\|_1+i \|a_1+n\|-(1+i)n. 
	\]	
	Similarly, if $a=a^*$ and $\|a\|\leq n$ then $\tau(a)=\|a+n\|_2^2-n$, which by the above argument  shows that $\tau$ is 2-definable. 
	
	Since $\|a\|_2=\tau(a^*a)^{1/2}$, the 2-norm is a 1-definable predicate. 
	
	The remaining parts of the lemma follow. 
\end{proof}

\begin{corollary}\label{C.affine.12}
	Two tracial von Neumann algebras are 1-elementarily equivalent if and only if they are 2-elementarily equivalent. 

	A class of tracial von Neumann algebras is 1-axiomatizable (in continuous logic) if and only if it is 2-axiomatizable (in continuous logic). 
\qed 
\end{corollary}

\subsection{Proof of Theorem~\ref{T.tensor}}
En route to the proof of Theorem~\ref{T.tensor} we prove the following (very likely well-known, yet not completely trivial) result. By Corollary~\ref{C.affine.12}, we do not need to indicate whether `elementarily equivalent' refers to the 1-logic or to the 2-logic  as discussed in the previous subsection.

\begin{proposition}\label{P.typeI}
	Suppose that $M$ and $M_1$ are elementarily equivalent tracial von Neumann algebras and one of them is of type I.  Then the other one is also of type I. 
	
	If in addition both $M$ and $M_1$ have separable predual, then they are isomorphic. In other words, the theory is a complete isomorphism invariant for separable tracial von Neumann algebas of  type I.
\end{proposition}

	Proposition~\ref{P.typeI} will follow from a more precise (and more obvious), statement (Lemma~\ref{L.rho}) given after a few clarifying remarks.

Note that being type I is not axiomatizable in language of tracial von Neumann algebras, since the category of type I tracial von Neumann algebras is not preserved under ultraproducts. (E.g., the ultraproduct of $\bbM_n(\bbC)$ for $n\in \bbN$ associated with a nonprincipal ultrafilter on $\bbN$ is an interesting  II$_1$ factor without property $\Gamma$.)
However, every tracial von Neumann algebra elementarily equivalent to $\bbM_n(\bbC)$ is isomorphic to it.

By the second part of Proposition~\ref{P.typeI},  type I tracial von Neumann algebras  behave similarly to compact metric structures, or to finite-dimensional \cstar-algebras (all of whose sorts are compact). 
More precisely, the second part of Proposition~\ref{P.typeI} is a poor man’s version of the fact that $\bbM_n(\bbC)\equiv A$ implies $\bbM_n(\bbC)\cong A$: Every tracial von Neumann algebra with separable predual elementarily equivalent to a given type I tracial von Neumann algebra $M$ with separable predual is isomorphic to it. In terminology of~\cite{Muenster}, being isomorphic to $M$ is \emph{separably axiomatizable}. In the standard model-theoretic terminology, the theory of~$M$ is $\aleph_0$-categorical (some authors write $\omega$-categorical, as the ordinal $\omega$ is routinely identified with the cardinal $\aleph_0$).

\begin{lemma}\label{L.rho}
If $M$ is a tracial von Neumann algebra, then there is a unique function 
\[
\rho_M\colon (\bbN\setminus \{0\})\times \bbN\to [0,1]
\] 
with the following properties. 
\begin{enumerate}
	\item  \label{1.P.typeI} $\sum_{m,n} \rho_M(m,n)\leq 1$, with the equality holding if and only if $M$ has type I. 
	\item \label{2.P.typeI} $\rho_M(m,n)\geq \rho_M(m,n+1)$ whenever $n\geq 1$. 
	\item \label{3.P.typeI} $M=\prod_{m\geq 1}{\bbM}_m(L^\infty(X_m,\mu_m))$, where $(X_m,\mu_m)$ is a measure space which has atoms of measure $\rho_M(m,n)$ for $n\geq 1$ (with multiplicities), and diffuse part of measure $\rho_M(m,0)$ (with $\mu(X_m)=\sum_n \rho_M(m,n)$).  
\end{enumerate}
Moreover, the function $\rho_M$ is computable from the theory of $M$. 
\end{lemma}

\begin{proof}
By the type decomposition of finite von Neumann algebras (\cite[\S V]{Tak:TheoryI}), $M$ is isomorphic to the direct sum $M_I\oplus M_{II}$, where $M_I$ is of type I and $M_{II}$ is of type II. 

By the same decomposition result, $M_I$ is of the form $\prod_{m\geq 1}{\bbM}_m(L^\infty(X_m,\mu_m))$ with $\sum_m \mu_m(X_m)=1$ (possibly with $\mu_m(X_m)=0$ for some $m$).  Since every finite measure space can be decomposed into diffuse and atomic part as specified, giving rise to $\rho_M$. Measures of the atoms are listed in decreasing order in order to assure \eqref{2.P.typeI}, securing the uniqueness of the function $\rho_M$. To be precise, let $Y_{m,n}$, for $n\in \bbN$, enumerate all atoms in the measure space $(X_m,\mu_m)$, listed in order of decreasing measure, with multiplicities. If there are only $k$ atoms,  then  let $Y_{m,n}=\emptyset$ for $n\geq k$. Finally let  $\rho_M(m,n)=\mu_m(Y_{m,n})$.   We therefore only need to explain how to determine $\rho_M$ from the theory of~$M$. 

First we use the fact that the center $Z(M)$ of a tracial von Neumann algebra $M$ is definable (this is essentially \cite[Lemma 4.1]{FaHaSh:Model1}). The proof shows that the lattice of projections in the center is also definable (this is not an immediate consequence of the fact that the set of projections is also definable, since by an unpublished result of Henson in continuous logic the intersection of definable sets is not necessarily definable). 

As observed in \cite[Theorem~2.5.1]{Muenster}, $m$-subhomogeneous \cstar-algebras are axiomatizable. This clearly extends to von Neumann algebras, by using the same formula. Therefore for every $m\geq 1$ the set (by $\tau$ we denote the distinguished tracial state of $M$)
\[
\{\tau(p)\mid p\in Z(M)\text{ is a projection  and $pMp$ is $m$-subhomogeneous}\}
\]
can be read off the theory of  $M$. The supremum of this set is equal to $\sum_{j\leq m}\sum_n \rho_M(j,n)$. Thus $\mu_m(X_m)$ is determined from $\Th(M)$. 

It remains to compute the measures of the atoms of each $\mu_m$. Again using the fact that the projections in $Z(M)$ form a definable set, these are the values of $\tau(p)$ where $p$ is a central projection such that $pMp\cong \bbM_m(\bbC)$. Multiplicities are handled similarly. 
\end{proof}

\begin{proof}[Proof of Proposition~\ref{P.typeI}]
Suppose that $M$ and $M_1$ are elementarily equivalent tracial von Neumann algebras and that $M$ has type I. Lemma~\ref{L.rho} implies that  $\rho_M=\rho_{M_1}$. By \eqref{1.P.typeI} of Lemma~\ref{L.rho} we have $\sum_{m,n} \rho_{M_1}(m,n)=\sum_{m,n} \rho_M(m,n)=1$, and therefore $M_1$ is also of type I.

Note that $\rho_M=\rho_{M_1}$ implies that the atomic parts of $M$ and $M_1$ are isomorphic. Since every diffuse abelian tracial von Neumann algebra is isomorphic to $L^\infty([0,1])$, if in addition to being of type I and elementarily equivalent both $M$ and $M_1$ have separable predual, then they are isomorphic.   
\end{proof}

The following well-known lemma uses the notion of the eq of a metric structure, studied in \cite[\S 3]{Muenster}. Briefly, if $M$ is a metric structure then $M^{\eq}$ is formed from $M$ by expanding it as follows (see \cite[\S 3.3]{Muenster} for details). First one adds all countable products of sorts in $M$ and equips them with natural product metric. Second, one adds all definable subsets of such products. Third, one takes quotients by all definable equivalence relations on such definable sets. The structure obtained in this way is denoted $M^{\eq}$. Its theory $T^{\eq}$ depends only $T=\Th(M)$, the category of models of $T^{\eq}$ is equivalent to the category of models of $T$,  and it is abstractly characterized as the largest conservative extension of the category of models of $T$ (this is \cite[Theorem 3.3.5]{Muenster}).  It is easier than its \cstar-algebraic analog, \cite[Lemma 3.10.2]{Muenster}.

\begin{lemma}\label{L.eq}
	If $N$ and $F$ are tracial von Neumann algebras and $F$ is finite-dimensional, then $F\bar\otimes N$ is in $N^{\eq}$. Thus, if $M$ is also a tracial von Neumann algebra such that $M\equiv N$ ($M\preceq M$), then $F\bar\otimes M \equiv F\bar\otimes N$ ($F\bar\otimes M\preceq F\bar \otimes N$).
\end{lemma}

\begin{proof}
	We first consider the case  $F={\bbM}_m(\bbC)$ for some $m\geq 1$. Then the distinguished tracial state of $F\bar\otimes N$ is $\sigma=\trm\otimes \tau$ (where $\trm$ is the normalized tracial state on $\bbM_m(\bbC)$ and $\tau$ is the distinguished tracial state of $N$). 
	Then the unit ball of $(F\bar\otimes N,\sigma)$ can be identified with a subset of the unit ball $N^{m^2}$ with the naturally defined matrix arithmetic  operations and tracial state $\sigma((a_{ij})_{i,j\leq m})=\sum_{i\leq m}\tau(a_{ii})$.
	
	For the general case, note that $F$ is the direct sum of full matrix algebras, $F=\bigoplus_{j\leq k} {\bbM}_{l(j)}(\bbC)$ and that its distinguished tracial state is a convex combination of $\tr_{l(j)}$, for $j\leq k$. By the argument similar to the one for ${\bbM}_m(\bbC)$, $F\bar\otimes N$ can be identified with $N^{\sum_{j\leq k}l(j)^2}$ with the appropriately defined arithmetic operations and distinguished tracial state. 
\end{proof}

\begin{proof}[Proof of Theorem~\ref{T.tensor}] We will prove the theorem for $\preceq$. The proof for $\equiv$ is analogous, and alternatively it follows by the fact that elementarily equivalent structures can be elementarily embedded into the same structure. 
	
	The result follows from the conjunction of the following two statements. 
	\begin{enumerate}
		\item \label{1.T.tensor}If $M,N,N_1$ are tracial von Neumann algebras, $N_1\preceq N$,  and $M$ has type I, then $M\bar\otimes N_1\preceq M\bar \otimes N$. 
		\item \label{2.T.tensor} If $M,N,N_1$ are tracial von Neumann algebras, $N_1\preceq N$,  and $N$ has type I, then $M\bar\otimes N_1\preceq M\bar \otimes N$. 
	\end{enumerate}
	
		\eqref{1.T.tensor} As in the proof of Proposition~\ref{P.typeI}, $M=\prod_{m\geq 1}\bbM_m(L^\infty(X_m,\mu_m))$. By Corollary \ref{FV-corollary 2}, $N_1\preceq N$ implies that  $N_1(m)=L^\infty(X_m,\mu_m)\bar\otimes N_1$ is an elementary submodel of $N(m)= L^\infty(X_m,\mu_m)\bar\otimes N$. 
		
		The matrix case of Lemma~\ref{L.eq} implies  $\bbM_m(N_1(m))\preceq \bbM_m(N(m))$ for all~$m$. Corollary \ref{FV-corollary 1} applied to the measure measure space $(\bbN, \mu)$ where $\mu$ is a probability Radon measure implies that  $M\bar\otimes N_1\preceq M\bar \otimes N$.
		
		\eqref{2.T.tensor} Analogously to case \eqref{1.T.tensor}, $N=\prod_{m\geq 1}\bbM_m(P_m)$ for some abelian von Neumann algebras $P_m$.   By 
		Proposition~\ref{P.typeI}, $N_1=\prod_{m\geq 1}\bbM_m(P_{1,m})$ with $P_{1,m}\equiv P_m$ for all $m$, and furthermore the algebras $P_m$ and $P_{1,m}$ have atoms and diffuse parts  of the same measure. This implies that  $P_m$ and $P_{1,m}$ are isomorphic. Therefore, $P_{1,m}\bar\otimes M\preceq P\bar\otimes M$.  Lemma~\ref{L.eq} implies that further tensoring with $\bbM_m(\bbC)$ preserves elementarity, and hence Corollary \ref{FV-corollary 1} implies that $\prod_{m\geq 1}\bbM_m(P_{1,m}\bar\otimes M)$ is an elementary submodel of $\prod_{m\geq 1}\bbM_m(P_{m}\bar\otimes M)$. That is, $M\bar\otimes N_1\preceq M\bar \otimes N$.
\end{proof}

\section{Concluding remarks}
\label{S.concluding}

In general, tensoring with strongly self-absorbing \cstar-algebras (\cite{ToWi:Strongly}) does not preserve elementarity (\cite[Proposition 6.2]{farah2022between}). All known examples of the failure of preservation of elementary equivalence by tensor products relied on failure of regularity properties of \cstar-algebras, such as definability of tracial states (\cite[Proposition~6.2]{farah2022between}) and stable rank being greater than 1 (\cite[Corollary~3.10.4]{Muenster}, using \cite[Theorem~3.1]{Phi:Exponential}).

However, known results imply that tensoring with a strongly self-ab\-sor\-bing \cstar-algebra $D$ preserves elementary equivalence in a large class of \cstar-algebras. 
It follows from  \cite[Corollary~2.7.2]{Muenster} that if $A$ is tensorially $D$-absorbing, then $A\preceq A\otimes D$ via the map that sends $a$ to $a\otimes 1_D$ for all $a\in A$ (since $D\cong D^{\bigotimes \bbN}$ by \cite[Proposition~1.9]{ToWi:Strongly}). 
Being tensorially $D$-absorbing is separably axiomatizable by \cite[Theorem~2.5.2]{Muenster}. (A property is \emph{separably axiomatizable} if there is a theory $T$ such that all separable models of $T$ satisfy this property.) Let $\bbTD$ be the theory of separable $D$-absorbing \cstar-algebras (see \cite[\S 2.1]{FaHaRoTi:Relative}). Models of $\bbTD$ are called \emph{potentially $D$-absorbing} (\cite[Definition~2.4]{FaHaRoTi:Relative}), and they  have the property that all of their separable elementary submodels are $D$-absorbing. 
 The Downward L\"owenheim--Skolem theorem thus implies that if $A$ is potentially $D$-absorbing and $A\equiv B$, then $A\otimes D\equiv B\otimes D$. 
 
 If   $A\preceq B$ and $A$ is potentially $D$-absorbing, then by identifying $A$ with $A\otimes 1_D$ in $A\otimes D$ we have $A\preceq B\otimes D$.
A slightly finer analysis using \cite[Lemma~1.4]{FaHaRoTi:Relative} applied to the inclusions $A\subseteq A\otimes D\subseteq B\otimes D$ (or directly using the Tarski-Vaught test) shows that $A\otimes D\preceq B\otimes D$.
 
 \begin{proposition} \label{prop. pot-D-abs}
  If  $A\preceq B$ and $A$ is potentially $D$-absorbing, then $A\otimes D\preceq B\otimes D$ via the natural embedding that fixes $1\otimes D$.
 \end{proposition} 

By a result of Connes, the hyperfinite II$_1$ factor $R$ is the only strongly self-absorbing II$_1$ factor. Potentially $R$-absorbing tracial von Neumann algebras are the McDuff factors, and an argument analogous to that of the previous paragraph shows that if $M\equiv N$ and $M$ is McDuff, then $M\botimes R\equiv N\botimes R$ and if $M\preceq N$ and $M$ is McDuff then $M\botimes R\preceq N\botimes R$.

It is not known whether tensor products of tracial von Neumann algebras preserve elementarity or not. The fact that theories of type~II$_1$ tracial von Neumann algebras do not admit quantifier elimination (see \cite{farah2023quantifier} and \cite{farah2023quantum} for a finer result) makes the question of preservation of elementarity more intricate. 

\begin{definition}
If $M=\int^\oplus_\Omega M_\omega\, d\mu(x)$ is a direct integral, then the \emph{distribution of the theories} in the measurable field $M_\omega$, for $\omega\in \Omega$, is the function $\alpha$ that to every $n\geq 1$ and every $n$-tuple of $\calL$-sentences $\bar\varphi=\langle \varphi_j: j\leq n\rangle$,  associates the distribution $\alpha_{\bar \varphi}\colon [0,1]^n\to [0,1]$ by 
\[
\alpha_{\bar \varphi}(\bar r)=\mu\{\omega \mid \varphi_j^{M_\omega}>r_j\text{ for all $j\leq n$}\}. 
\]
\end{definition}

Theorem~\ref{T.FV-integral} (together with Lemma~\ref{L.affine.12} that provides translation between languages $\LLTWO$ and $\LLONE$) has the following natural corollary. 

\begin{corollary} The theory of a direct integral $M=\int^\oplus_\Omega  M_\omega\, d\mu(\omega)$ is uniquely determined by the distribution of the theories in the measurable field $M_\omega$, for $\omega\in \Omega$. \qed 
\end{corollary}

The converse to this corollary is in general false, it is not possible to disintegrate the theory of a direct integral to recover the theories of $M_\omega$ for almost all $\omega$. However, every tracial  von Neumann algebra with separable predual admits a disintegration into a measurable field of factors that is essentially unique (\cite[\S IV]{Tak:TheoryI}). 
Confirming Conjecture~4.5 from the original version of the present paper, David Gao and David Jekel proved that if $M$ is a direct integral of a measurable field of II$_1$ factors then its theory uniquely determines the distribution of theories of II$_1$ factors in this measurable field. The proof uses a variant of \cite{ben2009randomizations}. Together with an easy Fubini-type argument, this theorem  implies the following.

\begin{corollary}\label{C1}
	Tensor products of tracial von Neumann algebras preserve elementary equivalence if and only if tensor products of II$_1$ factors preserve elementary equivalence.  \qed 
\end{corollary}

While a Feferman--Vaught-style theorem technically solves the problem of computing the theory of a given structure, it is desirable to have a more efficient procedure. 
In \cite{palyutin1} and \cite{palyutin2},  Palyutin isolated a class of so-called \emph{$h$-formulas} and shown that they satisfy a version of \L o\'s's theorem in every reduced product $M=\prod_{i\in I} M_i/\cI$ and that if in addition the Boolean algebra $ \cP(I)/\cI$ is atomless then every formula in the language of $M$ is equivalent to a Boolean combination of $h$-formulas. This has, for example, been used to provide a much simpler proof of $\aleph_1$-saturation of reduced products associated with a countable ideal in \cite{palmgren}. The analog of Palyutin’s theory for continuous logic has been developed in \cite{fronteau2023produits}. 

It would be desirable to develop analogous theory for direct integrals of structures in place of reduced products. We will now describe some partial results along these lines.   
In \cite{bagheri2010los}, Bagheri proved a preservation theorem for affine formulas under direct integrals. He introduced a variant of the continuous logic, nowadays known as the \emph{affine logic},  a systematic study of which is in  \cite{yaacov2024extremal} . Affine formulas are defined recursively starting from atomic formulas. Logical connectives are restricted to affine functions, while the role of quantifiers is still played by $\sup$ and $\inf$  (\cite[\S 2]{yaacov2024extremal}). Structures and interpretation of formulas are analogous to those in continuous logic.  The operation of taking direct integrals of measurable fields of affine structures preserves the affine theory (note that for tracial von Neumann algebras this is true only if they are considered with respect to the $\|\cdot\|_1$ norm, see Lemma~\ref{L.affine.12}). 
A preservation under direct integrals of tracial von Neumann algebras for certain convex formulas had been proven in \cite{farah3024tracial} in a work motivated by the need to systematize the theory of tracially complete \cstar-algebras (\cite{carrion2023tracially}).

Moving to the other important class of self-adjoint operator algebras, we ask whether there is a \cstar-algebraic analog of Theorem~\ref{T.tensor}? 
Tensor products by finite-dimensional \cstar-algebras preserve elementary equivalence. By the \cstar-al\-geb\-raic analog of Lemma~\ref{L.eq}, if $A$ and $F$ are \cstar-algebras such that~$F$ is finite-dimensional, then $A\otimes F$ belongs to $A^{\eq}$ (\cite[\S 3]{Muenster}) and therefore the theory of $A\otimes F$ can be computed from the theory of $A$. However, tensoring by $C([0,1])$ do not preserve elementary equivalence (\cite[Corollary~3.10.4]{Muenster}). To the best of our current (lamentably limited) understanding, the following would be a plausible analog of Theorem~\ref{T.tensor}. 

\begin{conjecture}
	Suppose that $A$ is a \cstar-algebra all of whose irreducible representations are finite-dimensional and the Gelfand spectrum of its center  is totally disconnected. If $B$ and $C$ are elementarily equivalent \cstar-algebras, then $A\otimes B$ and $A\otimes C$ are elementarily equivalent.  
\end{conjecture}

\bibliographystyle{amsplain}
\bibliography{ifmainbib}

\end{document}